%% file: multi.tex
\documentclass[a4paper]{article}

\input{def.tex}

\title{Consensus of Double Integrator Multiagent Systems under Nonuniform Sampling and Changing Topology}
\date{December 2022}
\author{Ufuk Sevim \thanks{Control and Automation Engineering Department, Istanbul Technical University, Istanbul, Turkey } \\ ufuk.sevim@itu.edu.tr \and Leyla Goren-Sumer \footnotemark[1] \\ leyla.goren@itu.edu.tr }

\begin{document}
	
	
	
	\maketitle
	\begin{abstract}
	
		This article considers consensus problem of multiagent systems with double integrator dynamics under nonuniform sampling. It is considered the maximum sampling time can be selected arbitrarily. Moreover, the communication graph can change to any possible topology as long as its associated graph Laplacian has eigenvalues in a given region, which can be selected arbitrarily. Existence of a controller that ensures consensus in this setting is shown when the changing topology graphs are balanced and has a spanning tree. Also, explicit bounds for controller parameters are given. A novel sufficient condition is given to solve the consensus problem based on making the closed loop system matrix a contraction using a particular coordinate system for general linear dynamics. It is shown that the given condition immediately generalizes to changing topology in the case of balanced topology graphs. This condition is applied to double integrator dynamics to obtain explicit bounds on the controller.
	\end{abstract}

%
	
	\input{introduction.tex}

	\input{preliminaries.tex}
	\input{problem.tex}
	\input{main.tex}
	\input{example.tex}
	\input{conclusion.tex}
	
%

	\def\url#1{}
	\bibliographystyle{plain}

\input{multi.bbl}
\end{document}

%% file: def.tex
\usepackage{amsmath, amssymb, amsthm}
\usepackage{xcolor}
\usepackage{algorithm}
\usepackage{algpseudocode}
\usepackage{graphicx}
\usepackage{caption}
\usepackage{subcaption}
\usepackage[numbers]{natbib}

\newtheorem{lemma}{Lemma}
\newtheorem{theorem}{Theorem}
{}
\newtheorem{corollary}{Corollary}{}
{}
\newtheorem{definition}{Definition}{}
{}
\newtheorem{eexample}{Example}{}

\def\N{\mathbb{N}}

\def\R{\mathbb{R}}
\def\C{\mathbb{C}}

\newcommand{\norm}[1]{\left\lVert#1\right\rVert}
\newcommand{\abs}[1]{\left\lvert#1\right\rvert}
\newcommand{\mat}[1]{\begin{bmatrix}#1\end{bmatrix}}

\newcommand{\overbar}[1]{\mkern 1.5mu\overline{\mkern-1.5mu#1\mkern-1.5mu}\mkern 1.5mu}

\def\diag{\operatorname{diag}}
\def\rank{\operatorname{rank}}

\def\hmax{\overbar{h}}

\def\one{\textbf{1}}

\newcommand\Item[1][]{%
  \ifx\relax#1\relax  \item \else \item[#1] \fi
  \abovedisplayskip=0pt\abovedisplayshortskip=0pt~\vspace*{-\baselineskip}}

%% file: introduction.tex
\section{Introduction}
The consensus problem of multi agent systems has received great attention in the last decade \citep{Qin2017}, mostly due to the broad application areas, such as mobile robot coordination \citep{Andreasson2014}, sensor network time synchronization \citep{Carli2008}, frequency synchronization in a microgrid \citep{Goldin2013} and many more \citep{Li2019}. Although general agent dynamics has got attention recently \citep{Zhang2017a, Gao2013, Zhang2014}, the study of double integrator agent dynamics is common \citep{Seuret2009a, Zhan2015a, Liu2010, Jesus2014, Lin2009a, Qin2012, Cheng2011}, partly because of their applicability to broad range of applications.

Many different aspects of the real world challenges are also studied extensively, such as consensus under switching topologies \citep{Zhang2017a}, consensus with communication delays \citep{Seuret2009a}, nonuniformly sampled-data consensus \citep{Zhan2015a}, asynchronous consensus \citep{Gao2010}, event-based consensus \citep{Zou2017}, partial state feedback consensus \citep{Abdessameud2013}, saturated input consensus \citep{Abdessameud2010} to name a few.

The study of nonuniform sampling and changing topology is particularly related to the topic of this paper, where sampling intervals and the communication topology cannot be determined reliably due to real-world constraints such as energy saving requirements, unreliable communication links and roaming agents. Although there are many works on either nonuniform sampling consensus or changing topology consensus, the number of studies that considers both aspects simultaneously is rather limited. In \citep{Zhan2015a} it is shown that consensus can be reached if the sampling intervals are constrained in a region that is determined by the max in-degree of the communication graph. However, the feasible sampling interval region gets smaller significantly when the max in-degree gets larger. In \citep{Ajwad2021} an Average Dwell Time approach is used to ensure the leader-following consensus. However, the set of switchable communication graphs must be known a priori which limits the usability of the method especially when the number of agents gets larger.

In this work we considered the double integrator consensus problem under nonuniform sampling where the maximum sampling time can be selected arbitrarily and the graph topology can change to any possible topology at any sampling time, provided that it is balanced, has a spanning tree and its associated graph Laplacian has eigenvalues in a given region, which can also be selected arbitrarily. We showed that a controller always exists that ensures consensus for this problem and give explicit bounds on the controller parameters. To the best of our knowledge, there is no study on double integrator consensus with these requirements.

We use the equivalent problem of stabilizing a subsystem of the overall system to solve the consensus problem for general linear dynamics, which is proven in \citep{Gao2013}. Therefore, existing stabilization results can be utilized to solve the consensus problem. In particular, we use a coordinate transformation that makes the overall system matrix a contraction for all possible sampling intervals. We also show that this approach immediately generalizes to changing topology in the case of balanced topology graphs. The obtained results for the general linear dynamics are then applied to the double integrator dynamics to find explicit bounds for the controller.

The rest of the paper is organized as follows. In section 2, we give necessary definitions and basic results that are used throughout the paper. In section 3, a novel sufficient condition is given to solve the consensus problem for general linear dynamics. In section 4, the given condition is applied to the double integrator dynamics to obtain explicit bounds for the controller parameters. In section 5, numerical examples are given to demonstrate the accuracy of the method. Finally, conclusions are given in section 6.

For a matrix $A \in \C^{n \times n}$, $\sigma(A)$ denotes the set of singular values and $\bar{\sigma}(A)$ denotes the maximum singular value of $A$. We say $A$ is a contraction if $\bar{\sigma}(A) < 1$. $\abs{\cdot}$ and $\norm{\cdot}$ represents some vector and matrix norms respectively. We use subscripts $i,j$ for the indices and $k \in \N$ for the discrete time instance where $\N$ represents the nonnegative integers.

%% file: preliminaries.tex
\section{Preliminaries}
\subsection{Graph Theory}
The network topology of a multiagent system with $N$ agents is represented by a graph which is defined as a pair $G = (V, E)$ where $V = \{v_1, \dots, v_N\}$ is the set of nodes and $E \subseteq V \times V$ is the set of edges. The nodes of the graph correspond to the agents and an edge $(v_i, v_j) \in E$ denotes a directed information flow from agent $j$ to agent $i$. We assume that the topology graph of the multiagent system is simple, i.e. there are no self-loops, that is $(v_i, v_i) \notin E, \forall i$, and there are no multiple edges between pair of nodes.

The neighbor set of node $i$ is the index set defined as $\mathcal{N}_i := \{j ~|~ (v_i, v_j) \in E \}$, that is only the nodes with edges incoming to $v_i$ are considered neighbors. A directed tree is a graph where every node has exactly one neighbor except the so-called root node, which has no neighbors. If a subset of the edges of a graph forms a directed tree, then the graph is said to have a spanning tree. If a graph has a spanning tree, then all of the nodes are reachable from the root by following the (directed) edges.

The weighted adjacency matrix $W = (w_{ij}) \in \R^{N \times N}$ represents the nonnegative weights associated with edges, where $w_{ij} > 0$ if $(v_i, v_j) \in E$ and $w_{ij} = 0$ if $(v_i, v_j) \notin E$. A graph is called balanced if $w_{ij} = w_{ji}, \forall i,j$, that is $W = W^T$ is symmetric.

The weighted in-degree of node $i$ is defined as $d_i := \sum_{j \in \mathcal{N}_i} w_{ij}$. The Laplacian matrix of a graph is defined as $L := D - W$ where $D := \diag\{d_i\}$. Obviously, $L \one = 0$ where $\one^T := \mat{1 & \dots & 1}$. Therefore 0 is an eigenvalue of $L$. By the Gershgorin circle theorem all eigenvalues of $L$ lies in the disc $\{z \in \C ~|~ \abs{z - d_{\max}} \leq d_{\max}\}$ where $d_{\max} := \max_i \{d_i\}$. Therefore, $\operatorname{Re}(\lambda) \geq 0$ for all eigenvalues $\lambda$ of $L$.

\begin{theorem}[\cite{Lewis2014}]
	The zero eigenvalue of $L$ is simple, i.e. it has algebraic multiplicity 1, if and only if graph $G$ has a spanning tree.
\end{theorem}

\subsection{Stability of Time-Varying Discrete Systems}
Stability of nonuniformly sampled systems can be analyzed by investigating the time-varying discrete systems. It is well-known that the stability conditions for these systems are not trivial \citep{Liberzon1999}. A sufficient stability result based on induced matrix norms is given in this subsection.

\begin{definition}
	The system
	\begin{equation}
		\label{eq:sd_ol}
		x_{k+1} = F_k x_k
	\end{equation}
	is globally asymptotically stable if, for any $x_0 \in \R^n$
	\begin{equation*}
		\lim_{k \to \infty} \abs{x_k} = 0
	\end{equation*}
	for some vector norm $\abs{\cdot}$.
\end{definition}

\begin{theorem}
	\label{thm:stability_norm}
	If there exists a matrix norm $\norm{\cdot}$ induced over some vector norm $\abs{\cdot}$ such that $\norm{F_k} < 1, \forall k \in \N$, then the system \eqref{eq:sd_ol} is globally asymptotically stable.
\end{theorem}

\begin{proof}
	Let $V(x_k) := \abs{x_k}$. Then,
	\[ V(x_{k+1}) = \abs{x_{k+1}} \leq \norm{F_k} \abs{x_k} < \abs{x_k} = V(x_k). \]
	Hence $V(\cdot)$ is a Lyapunov function for the system \eqref{eq:sd_ol}.
\end{proof}

\begin{lemma}
	\label{lem:norm_similarity}
	Let $\norm{\cdot}$ be a matrix norm induced over some vector norm $\abs{\cdot}$. Then for any $A \in \R^{n \times n}$
	\begin{equation}
		\norm{A}_T := \norm{T^{-1} A T}
	\end{equation}
	is an induced norm over the vector norm
	\begin{equation}
		\abs{x}_T := \abs{T^{-1} x}
	\end{equation}
	where $x \in \R^n$ and $T \in \R^{n \times n}$ is an invertible matrix. 
\end{lemma}

\begin{proof}
	It is easy to see that $\abs{\cdot}_T$ is a vector norm. Then,
	\[
	\sup_{x \in \R^n} \frac{\abs{A x}_T}{\abs{x}_T} = \sup_{y \in \R^n} \frac{\abs{A T y}_T}{\abs{T y}_T} = \sup_{y \in \R^n} \frac{\abs{T^{-1} A T y}}{\abs{T^{-1} T y}} = \norm{A}_T \qedhere
	\]
\end{proof}

\begin{corollary}
	\label{cor:stability_sv}
	If there exists an invertible matrix $T \in \R^{n \times n}$ such that
	\begin{equation}
		\label{ineq:singular_value}
		\bar{\sigma}(T^{-1} F_k T) < 1, ~~ \forall k \in \N,
	\end{equation}
	then the system \eqref{eq:sd_ol} is globally asymptotically stable, where $\bar{\sigma}(\cdot)$ denotes the maximum singular value.
\end{corollary}

\begin{proof}
	It is well-known that the maximum singular value of a matrix is the induced matrix norm over Euclidean vector norm. Then the result follows from Lemma \ref{lem:norm_similarity}.
\end{proof}


\subsection{Stability of Nonuniformly Sampled Systems}
Consider the continuous time system
\begin{equation}
	\label{eq:cont_ol}
	\dot{x}(t) = Ax(t) + Bu(t)
\end{equation}
where $x(t) \in \R^n$, $u(t) \in \R^m$, $A \in \R^{n \times n}$, $B \in \R^{n \times m}$, $\rank B = m \leq n$, and $(A, B)$ is stabilizable.

Define the sequence of sampling time instances $\{t_k\}_{k \in \N}$ where
\[ 0 = t_0 < t_1 < \dots < t_k < \dots \]
with $\lim_{k \to \infty} t_k = \infty$. We assume that sampling intervals $h_k := t_{k+1} - t_k$ are bounded, i.e. $h_k \in (0, \hmax), \forall k \in \N$ for some $\hmax > 0$.

Let $u(t) := K x(t_k), \forall t \in [t_k, t_{k+1})$ where $K \in \R^{m \times n}$ is the state feedback controller. Then the closed loop system becomes
\begin{equation}
	\label{eq:sd_cl}
	\dot{x}(t) = A x(t) + B K x(t_k), ~~ \forall t \in [t_k, t_{k+1}).
\end{equation}

The following result is standard:
\begin{lemma}
	\label{lem:discretization}
	Let $x_0 := x(t_0)$. Then the solutions of \eqref{eq:sd_cl} and
	\begin{equation}
		\label{eq:discrete_cl}
		x_{k+1} = \left( F(h_k) + G(h_k) K \right) x_k
	\end{equation}
	are the same at the sampling instances, i.e. $x_k = x(t_k), \forall k \in \N$, where
	\begin{equation} \label{eq:FhGh}
		F(h) := e^{A h} ~~ \text{and} ~~ G(h) := \left( \int_0^{h} e^{A \tau} d \tau \right) B.
	\end{equation}
\end{lemma}

\begin{proof}
	We prove with induction. $x(t_0) = x_0$ by assumption. Assume that $x(t_k) = x_k$, so
	\begin{align*}
		x(t_{k+1}) &= e^{A (t_{k+1} - t_k)} x(t_k) + \int_{t_k}^{t_{k+1}} e^{A(t_{k+1} - \eta)} B u(\eta) d \eta \\
		&= e^{A h_k} x(t_k) + \left( \int_0^{h_k} e^{A \tau} d \tau \right) B u(t_k) \\
		&= F(h_k) x_k + G(h_k) K x_k \\
		&= x_{k+1} \qedhere
	\end{align*}
\end{proof}

The following corollary follows as a result of Lemma \ref{lem:discretization} and Corollary \ref{cor:stability_sv}:
\begin{corollary} \label{cor:stability_nonuniform}
	The closed loop system \eqref{eq:sd_cl} is globally asymptotically stable for an arbitrary selection of sampling intervals $h_k \in (0, \hmax), \forall k \in \N$ if there exists an invertible $T \in \R^{n \times n}$ such that
	\begin{equation} \label{ineq:stability_sv}
		\bar{\sigma} \left( T^{-1} \left( F(h) + G(h) K \right) T \right) < 1, ~~~~ \forall h \in (0, \hmax).
	\end{equation}
\end{corollary}

Note that if $T$ satisfies \eqref{ineq:stability_sv}, then $\overbar{T} := TV$ also satisfies \eqref{ineq:stability_sv} for any orthogonal matrix $V^T V = I$.

\subsection{Bounds on Singular Values}
A Gershgorin type bound is given in \citep{Qi1984} along with some other simple estimates for singular values. Since we are only interested in the maximum singular value of a square matrix, a corollary is given as follows:
\begin{corollary}[of Theorem 2 in \citep{Qi1984}] \label{cor:gershgorin}
	Let $A = (a_{ij}) \in \C^{n \times n}$. Then,
	\begin{equation}
		\bar{\sigma}(A) \leq \max_i \left\{ s_i \right\}
	\end{equation}
	where $s_i := \max \{ r_i, c_i \}$, $r_i := \sum_{j=1}^{n} \abs{a_{ij}}$ and $c_i := \sum_{j=1}^{n} \abs{a_{ji}}$.
\end{corollary}

Corollary \ref{cor:gershgorin} is a fairly standard result and one can easily show this using other methods such as Hölder's inequality. We extend this result to block matrices using the proof idea in \citep{Qi1984}.

\begin{lemma} \label{lem:gershgorin_block}
	Let $A = (A_{ij}) \in \C^{nN \times nN}$ be a block matrix where each block is the same size, i.e. $A_{ij} \in \C^{n \times n}$. Then,
	\begin{equation}
		\bar{\sigma}(A) \leq \max_i \left\{ s_i \right\}
	\end{equation}
	where $s_i := \max \{ r_i, c_i \}$, $r_i := \sum_{j=1}^{n} \bar{\sigma}(A_{ij})$ and $c_i := \sum_{j=1}^{n} \bar{\sigma}(A_{ji})$.
\end{lemma}
\begin{proof}
	Let $\sigma$ be a singular value of $A$. Then there exists nonzero vectors
	\[x = \mat{x_1^T & \dots & x_{N}^T}^T \in \C^{nN} ~~\text{and}~~ y = \mat{y_1^T & \dots & y_{N}^T}^T \in \C^{nN} \]
	such that
	\[ \mat{A_{11} & \dots & A_{1N} \\ \vdots & \ddots & \vdots \\ A_{N1} & \dots & A_{NN}} \mat{x_1 \\ \vdots \\ x_N} = \sigma \mat{y_1 \\ \vdots \\ y_N} ~~\text{and}~~ \mat{A_{11}^* & \dots & A_{N1}^* \\ \vdots & \ddots & \vdots \\ A_{1N}^* & \dots & A_{NN}^*} \mat{y_1 \\ \vdots \\ y_N} = \sigma \mat{x_1 \\ \vdots \\ x_N} \]
	where $x_i, y_i \in \C^n, i = 1,\dots,N$ and $A^*$ is the conjugate transpose of $A$. Let $\alpha := \max \{ \abs{x_1}, \dots, \abs{x_N}, \abs{y_1}, \dots, \abs{y_N} \}$ where $\abs{\cdot}$ is the Euclidean norm. If $\alpha = \abs{y_i}$ for some $i$, then
	\[ \sigma y_i = \sum_{j=1}^{N} A_{ij} x_j ~~\text{implies}~~ \sigma \leq \sum_{j=1}^{n} \bar{\sigma}(A_{ij}) = r_i \]
	If $\alpha = \abs{x_i}$ for some $i$, then
	\[ \sigma x_i = \sum_{j=1}^{N} A_{ji}^* y_j ~~\text{implies}~~ \sigma \leq \sum_{j=1}^{n} \bar{\sigma}(A_{ji}) = c_i \]
	In any case $\sigma \leq \max\{r_i, c_i\} = s_i$. Since this holds for any singular value of $A$, we can conclude the proof.
\end{proof}

\begin{lemma} \label{lem:block_singular}
	Let $A, B \in \R^{n\times n}$ and
	\[ C := \mat{A & -B \\ B & A}. \]
	Then, $\sigma(C) = \sigma(A - jB) \cup \sigma(A + jB)$ where $\sigma(\cdot)$ is the set of singular values and $j = \sqrt{-1}$.
\end{lemma}
\begin{proof}
	It is well-known that
	\[ \det(C) = \det(A - jB) \det(A + jB). \]
	Noting that $CC^T$ has the same form as $C$, we have
	\begin{align*}
		\det(CC^T - \lambda I) &= \det \left( \mat{A & -B \\ B & A} \mat{A^T & B^T \\ -B^T & A^T} - \lambda I \right) \\
		&= \det \mat{AA^T + BB^T - \lambda I & AB^T - BA^T \\ BA^T - AB^T & AA^T + BB^T - \lambda I} \\
		&= \det \left( (AA^T + BB^T - \lambda I) - j(BA^T - AB^T) \right) \\
		&~~~~~~~~~~\det \left( (AA^T + BB^T - \lambda I) + j(BA^T - AB^T) \right) \\
		&= \det \left( (A-jB) (A^T+jB^T) - \lambda I \right) \\
		&~~~~~~~~~~\det \left( (A+jB) (A^T-jB^T) - \lambda I \right)
	\end{align*}
	The result follows from the definition of the singular values.
\end{proof}

%% file: problem.tex
\section{Problem Formulation}
We first define the problem for fixed network topology and then we show that our approach immediately generalizes to changing topology in the case of balanced topology graphs. Consider the $N$-agent system on the network topology graph $G$ with identical dynamics
\begin{equation} \label{eq:cont_node}
	\dot{x}_i(t) = Ax_i(t) + Bu_i(t), ~~~~ i = 1,2,\dots,N
\end{equation}
where $x_i(t) \in \R^n$ is the state of agent $i$ and $u_i(t) \in \R^m$ is its control input. The control inputs $u_i(t)$ is said to solve the consensus problem asymptotically if they drive all the states to the same values for any initial states, i.e.,
\[ \lim_{t \to \infty} \left(x_i(t) - x_j(t)\right) = 0, ~~~~ \forall i,j = 1,\dots,N. \]

We select the control input as distributed local state feedback law, i.e. every agent only uses the information from its neighbors, under nonuniform sampling with controller matrix $K \in \R^{m \times n}$
\begin{equation} \label{eq:state_feedback}
	u_i(t) = K \sum_{j \in \mathcal{N}_i} w_{ij}(x_j(t_k) - x_i(t_k)), ~~~~ \forall t \in [t_k, t_{k+1})
\end{equation}
where $w_{ij}$ are the elements of the weighted adjacency matrix of the topology graph.

The overall closed-loop graph dynamics can be written as \citep[see][]{Lewis2014}
\begin{equation} \label{eq:cl_cont}
	\dot{x}(t) = (I_N \otimes A) x(t) - (L \otimes BK) x(t_k)
\end{equation}
where $x := \mat{x_1^T & x_2^T & \dots & x_N^T}^T \in \R^{nN}$ is the overall state vector, $L$ is the graph Laplacian matrix and $\otimes$ is the Kronocker product.

Let $z(t) := (M^{-1} \otimes I_n) x(t)$ where $M := \mat{\one & \overbar{M}}$ so that
\[ M^{-1} L M = \mat{0 & \ell^T \\ 0 & \overbar{L}} \]
where $\one := \mat{1 & 1 & \dots & 1}^T$. So, the closed loop system dynamics can be separated as
\begin{align}
	\label{eq:cl_cont_z_1} \dot{z}_1(t) &= A z_1(t) - (\ell^T \otimes BK) \xi(t_k) \\
	\label{eq:cl_cont_xi} \dot{\xi}(t) &= (I_{N-1} \otimes A) \xi(t) - (\overbar{L} \otimes BK) \xi(t_k)
\end{align}
where $z = \mat{z_1^T & \xi^T}^T$ and $\xi := \mat{z_2^T & \dots & z_N^T}^T$.

\begin{lemma}[\cite{Gao2013}] \label{lem:consensus_stability}
	Control law \eqref{eq:state_feedback} solves the consensus problem asymptotically if and only if system \eqref{eq:cl_cont_xi} is asymptotically stable. In other words
	\[ \lim_{t \to \infty} \xi(t) = 0 \iff \lim_{t \to \infty} (x_i(t) - x_j(t)) = 0 ~~~~ i,j = 1,2,\dots,N \]
\end{lemma}

Using Lemma \ref{lem:discretization}, stability of \eqref{eq:cl_cont_xi} can be analyzed using its discretized model.
\begin{align}
	\xi_{k+1} &= e^{(I_{N-1} \otimes A) h_k} \xi_k - \left( \int_0^{h_k} e^{(I_{N-1} \otimes A) \tau} d \tau \right) \left(\overbar{L} \otimes BK\right) \xi_k \nonumber \\
	&= \left( I_{N-1} \otimes e^{A h_k} \right) \xi_k - \left( I_{N-1} \otimes \left( \int_0^{h_k} e^{A \tau} d \tau \right) \right) \left(\overbar{L} \otimes BK\right) \xi_k \nonumber \\
	&= \left( I_{N-1} \otimes e^{A h_k} - \overbar{L} \otimes \left( \int_0^{h_k} e^{A \tau} d \tau \right) BK \right) \xi_k \nonumber \\
	&= \left( I_{N-1} \otimes F(h_k) - \overbar{L} \otimes G(h_k) K \right) \xi_k \label{eq:sys_all_dt}
\end{align}

Using Lemma \ref{lem:consensus_stability} and Corollary \ref{cor:stability_nonuniform}, we can conclude that if there exists an invertible $T \in \R^{n \times n}$ such that the transformed closed loop system matrix
\begin{align*}
	\hat{\Phi}(h) &:= \left(I_{N-1} \otimes T^{-1}\right) \left( I_{N-1} \otimes F(h) - \overbar{L} \otimes G(h) K \right) \left(I_{N-1} \otimes T\right) \\
	&\phantom{:}= I_{N-1} \otimes T^{-1} F(h) T - \overbar{L} \otimes T^{-1} G(h) K T\\
	&\phantom{:}= I_{N-1} \otimes \hat{F}(h) - \overbar{L} \otimes \hat{G}(h) \hat{K}
\end{align*}
is a contraction, i.e. $\bar{\sigma} (\hat{\Phi}(h)) < 1$, for all $h \in (0, \hmax)$, then control law \eqref{eq:state_feedback} solves the consensus problem asymptotically for arbitrary selection of sampling intervals, where
\[ \hat{F}(h) := T^{-1} F(h) T, ~~ \hat{G}(h) := T^{-1} G(h) ~~\text{and}~~ \hat{K} := KT. \]

Without losing generality, we can assume that $\overbar{L}$ is in Jordan form whose off-diagonals are arbitrarily small. Therefore $\hat{\Phi}(h) = \diag\{\hat{J}_i(h)\}$ is in the block diagonal form where each block is as follows:
\[ \hat{J}_i(h) := \mat{
	\hat{S}(h, \lambda_i) & \delta \hat{G}(h) \hat{K} & 0 & \dots & 0 \\
	0 & \hat{S}(h, \lambda_i) & \delta \hat{G}(h) \hat{K} & \dots & 0 \\
	0 & 0 & \hat{S}(h, \lambda_i) & \dots & 0 \\
	\vdots & \vdots & \vdots & \ddots & \vdots \\
	0 & 0 & 0 & \dots & \hat{S}(h, \lambda_i)
} \]
where $\lambda_i$ is a possibly complex eigenvalue of $\overbar{L}$ (see Lemma \ref{lem:block_singular} for real block diagonal case), $\hat{S}(h, \lambda_i) := \hat{F}(h) - \lambda_i \hat{G}(h) \hat{K}$ and $\delta > 0$ is arbitrarily small. Using Lemma \ref{lem:gershgorin_block}, we can see that
\[ \bar{\sigma} \left( \hat{\Phi}(h) \right) = \max_i \left\{ \bar{\sigma} \left( \hat{J}_i(h) \right) \right\} \leq \max_i \left\{ \bar{\sigma} \left( \hat{S}(h, \lambda_i) \right) + \delta \bar{\sigma} \left( \hat{G}(h) \hat{K} \right) \right\} < 1 \]
is sufficient to solve the consensus problem. Since $\delta$ is arbitrarily small and does not depend on the selection of $T$ or $K$, we can conclude that
\begin{equation} \label{ineq:main_consensus}
	\max_i \left\{ \bar{\sigma} \left( \hat{S}(h, \lambda_i) \right) \right\} < 1
\end{equation}
is sufficient. We summarize all the discussions above with the following theorem:
\begin{theorem} \label{thm:main_general}
	Let $(A,B)$ be stabilizable system dynamics of the agents and the topology graph has a spanning tree. Then, control law \eqref{eq:state_feedback} solves the consensus problem asymptotically if there exists an invertible matrix $T \in \R^{n \times n}$ such that
	\begin{equation}
		\bar{\sigma} \left( T^{-1} S(h, \lambda_i) T \right) < 1, ~~ \forall h \in (0, \hmax),~~ \forall i \in \{2, \dots, N\}
	\end{equation}
	where
	\begin{equation}
		S(h, \lambda) := e^{A h} - \lambda \left( \int_0^{h} e^{A \tau} d \tau \right) BK
	\end{equation}
	and $\lambda_i$ are the eigenvalues of the Laplacian matrix of the topology graph.
\end{theorem}

Consider the case of changing topology at each sampling instance with graph Laplacian matrices $L_k$. From the proof given in \citep{Gao2013}, it is easy to see that Lemma \ref{lem:consensus_stability} still applies. In the case of balanced graphs, i.e. $L_k = L_k^T$, one can select $M$ such that $\overbar{L}_k = \overbar{L}_k^T$ for all $k$. Indeed, this is true when $\overbar{M}^T \overbar{M} = I$ and $\overbar{M}^T \one = 0$.

Any symmetric matrix is orthogonally diagonalizable, so there exists orthogonal matrices $U_k^{-1} = U_k^T$ such that $U_k^T \overbar{L}_k U_k$ is diagonal. Also, multiplication by an orthogonal matrix does not change the singular values of a matrix. Since $U_k \otimes I_n$ is also orthogonal, we can assume $\overbar{L}_k$ is diagonal for all $k$. Therefore, we can reach the following result.

\begin{theorem} \label{thm:main_symmetric}
	Let $(A,B)$ be stabilizable system dynamics of the agents. Assume that the changing topology graphs are balanced and have a spanning tree at each nonuniform sampling instance. Then, control law \eqref{eq:state_feedback} solves the consensus problem asymptotically if there exists an invertible matrix $T \in \R^{n \times n}$ such that
	\begin{equation}
		\bar{\sigma} \left( T^{-1} S(h, \lambda) T \right) < 1, ~~ \forall h \in (0, \hmax),~~ \forall \lambda \in [\lambda_2, \lambda_N]
	\end{equation}
	where $\displaystyle \lambda_2 = \inf_{k,i} \lambda_i (\overbar{L}_k)$ and $\displaystyle \lambda_N = \sup_{k,i} \lambda_i (\overbar{L}_k)$.
\end{theorem}

%% file: main.tex
\section{Main Results}
Although Theorems \ref{thm:main_general} and \ref{thm:main_symmetric} are given for general system dynamics, it is hard to find $T$ and $K$ that satisfies it in general. However, we give explicit bounds on the controller parameters in the case of double integrator dynamics and balanced network topologies. The explicit bounds depend on the maximum sampling interval and minimum and maximum eigenvalues of all possible Laplacian matrices. These parameters can be selected arbitrarily depending on the application. Also, it is shown that a controller always exists for any selection of these parameters, as long as the network graph has a spanning tree, which is known to be a necessary condition for consensusability \citep{Ma2010}.


Consider double integrator agent dynamics, i.e.,
\[ A := \mat{0 & 1 \\ 0 & 0} ~~\text{and}~~ B := \mat{0 \\ 1} \]
Then,
\[ S(h,\lambda) = \mat{1 & h \\ 0 & 1} - \lambda \mat{\frac{h^2}{2} \\ h} K \]
Also let the eigenvalues of the Laplacian matrix of the changing topology graphs be in the interval $[\lambda_2, \lambda_N]$.

We propose the following similarity transformation matrix
\begin{equation}
	T := \mat{-\mu_1 & -\mu_2 \\ 1 & 1} \mat{1 & 1 \\ -1 & 1} = \mat{\mu_2 - \mu_1 & -(\mu_2 + \mu_1) \\ 0 & 2}
\end{equation}
where $0 < \mu_1 < \mu_2$. Therefore,
\[ \hat{S} := T^{-1} S(h,\lambda) T = \mat{1 - \frac{h \lambda \gamma k_1}{2} & \frac{2 h}{\mu_2 - \mu_1} - \frac{h \lambda \gamma k_2}{2} \\ - \frac{h \lambda k_1}{2} & 1 - \frac{h \lambda k_2}{2}} \]
where
\[ \gamma := \frac{\mu_1 + \mu_2 + h}{\mu_2 - \mu_1} ~~\text{and}~~ \hat{K} := KT = \mat{k_1 & k_2}. \]
Using Corollary \ref{cor:gershgorin} we need
\begin{equation} \label{ineq:main}
	\max\{ \abs{s_{11}} + \abs{s_{12}}, \abs{s_{11}} + \abs{s_{21}}, \abs{s_{22}} + \abs{s_{12}}, \abs{s_{22}} + \abs{s_{21}} \} < 1
\end{equation}
to solve the consensus problem where $\hat{S} = (s_{ij})$. From \eqref{ineq:main}, we have 4 inequalities to be satisfied, since each of the sums must be smaller than 1.

First note that we have $h,\lambda > 0$, $\gamma > 1$ and $\mu_2 > \mu_1 > 0$. We also need $k_1, k_2 > 0$, as otherwise $s_{11}, s_{22} > 1$ so we cannot use the condition \eqref{ineq:main}.

We also require $s_{11}, s_{22} > 0$, since otherwise the lower limit for $k_1$ and $k_2$ can be large for small values of $h$. So, immediately we have the following conditions to be satisfied:
\begin{align*}
	s_{11} &= 1 - \frac{h \lambda \gamma k_1}{2} > 0 &&\Rightarrow \frac{2}{h \lambda \gamma} > k_1 \\
	s_{22} &= 1 - \frac{h \lambda k_2}{2} > 0 &&\Rightarrow \frac{2}{h \lambda} > k_2
\end{align*}

Since $s_{21}$ is already negative, we need to satisfy:
\begin{align*}
	\abs{s_{11}} + \abs{s_{21}} &= 1 - \frac{h \lambda (\gamma - 1) k_1}{2} < 1 &&\Rightarrow \text{Already satisfied} \nonumber \\
	\abs{s_{22}} + \abs{s_{21}} &= 1 - \frac{h \lambda}{2} (k_2 - k_1) < 1 &&\Rightarrow k_2 - k_1 > 0
\end{align*}

For $s_{12} > 0$ one can show that the resulting inequalities are inconsistent. Assuming $s_{12} < 0$ we have the final set of conditions to be satisfied as follows:
\begin{align*}
	s_{12} &= \frac{2 h}{\mu_2 - \mu_1} - \frac{h \lambda \gamma k_2}{2} < 0 &&\Rightarrow k_2 > \frac{4}{\lambda (\mu_1 + \mu_2 + h)} \\
	\abs{s_{11}} + \abs{s_{12}} &= 1 - \frac{2 h}{\mu_2 - \mu_1} + \frac{h \lambda \gamma}{2} (k_2 - k_1) < 1 &&\Rightarrow \frac{4}{\lambda (\mu_1 + \mu_2 + h)} > k_2 - k_1 \\
	\abs{s_{22}} + \abs{s_{12}} &= 1 - \frac{2 h}{\mu_2 - \mu_1} + \frac{h \lambda (\gamma - 1) k_2}{2} < 1 &&\Rightarrow \frac{4}{\lambda (2 \mu_1 + h)} > k_2
\end{align*}


To solve the problem for all possible $(h, \lambda)$, we must minimize the left sides and maximize the right sides of the inequalities. This can be done by writing $h = \hmax$ and $\lambda = \lambda_N$ for the left sides and $h = 0$ and $\lambda = \lambda_2$ for the right sides. Putting all of them together we finally obtain the following inequalities:
\begin{align}
	\frac{2 (\mu_2 - \mu_1)}{\hmax \lambda_N (\mu_1 + \mu_2 + \hmax)} &> k_1 > 0 \label{ineq:main1} \\
	\min \left\{ \frac{2}{\hmax \lambda_N}, \frac{4}{\lambda_N (2 \mu_1 + \hmax)} \right\} &> k_2 > \frac{4}{\lambda_2 (\mu_1 + \mu_2)} \label{ineq:main2} \\
	\frac{4}{\lambda_N (\mu_1 + \mu_2 + \hmax)} &> k_2 - k_1 > 0 \label{ineq:main3}
\end{align}

We can conclude that the controller $K = \mat{k_1 & k_2} T^{-1}$ solves the consensus problem for the double integrator agent dynamics for arbitrary selection of sampling intervals and changing balanced topologies such that $h_k \in (0, \hmax)$ and $\lambda \in [\lambda_2, \lambda_N]$ when $k_1,k_2$ satisfies \eqref{ineq:main1}-\eqref{ineq:main3}.

Now we can analyze the consistency of these inequalities.
\begin{lemma} \label{lem:ineq_consistency}
	Let $a,b,c,d > 0$ and
	\begin{align*}
		a &> k_1 > 0 \\
		b &> k_2 > c \\
		d &> k_2 - k_1 > 0
	\end{align*}
	Then the inequalities are consistent, i.e. there exists $k_1,k_2$ satisfying them, if and only if $b > c$ and $a + d > c$.
\end{lemma}
\begin{proof}
	($\Rightarrow$) Assume the inequalities are consistent, then obviously $b>c$. By adding the first and third one we obtain $a+d> k_2 >c$.
	
	($\Leftarrow$) Let $b > c$ and $a + d > c$. If $d \geq c$, select $k_1 = \alpha$ where $\min\{a,c\} > \alpha > 0$, then by the second and third inequalities we have
	\[\min\{d + \alpha, b\} > k_2 > \max\{c, \alpha\} = c\]
	and the inequalities are consistent. If $c > d$, select $k_2 - k_1 = d - \epsilon$ where $a + d - c > \epsilon > 0$, then by the first and second inequalities we have 
	\[ \min\{a, b-d+\epsilon \} > k_1 > \max\{c-d+\epsilon, 0\} = c-d+\epsilon \]
	and the inequalities are consistent.
\end{proof}

\begin{corollary} \label{cor:ineq_consistent}
	Inequalities \eqref{ineq:main1}-\eqref{ineq:main3} are consistent if and only if
	\begin{equation} \label{ineq:necessary}
		\frac{\mu_1 + \mu_2}{\hmax + \max\{\hmax, 2 \mu_1\}} > \frac{\lambda_N}{\lambda_2}
	\end{equation}
\end{corollary}
\begin{proof}
	Let
	\begin{align*}
		a &:= \frac{2 (\mu_2 - \mu_1)}{\hmax \lambda_N (\mu_1 + \mu_2 + \hmax)},~~ &&b := \frac{4}{\lambda_N \left( \hmax + \max\{\hmax, 2 \mu_1\} \right)},\\
		c &:= \frac{4}{\lambda_2 (\mu_1 + \mu_2)},~~ &&d := \frac{4}{\lambda_N (\mu_1 + \mu_2 + \hmax)}
	\end{align*}
	According to Lemma \ref{lem:ineq_consistency}, inequalities \eqref{ineq:main1}-\eqref{ineq:main3} are consistent if and only if $b > c$ and $a+d>c$.
	It is easy to see that $b > c$ is equivalent to \eqref{ineq:necessary}. We are going to show that $a + d \geq b$ to conclude the proof. Therefore, we need
	\begin{align*}
		\frac{4 \hmax + 2 (\mu_2 - \mu_1)}{\hmax \lambda_N (\mu_1 + \mu_2 + \hmax)} &\geq \frac{4}{\lambda_N \left( \hmax + \max\{\hmax, 2 \mu_1\} \right)} \\
		\max\{\hmax, 2 \mu_1\} &\geq \frac{2 \hmax (\mu_1 + \mu_2 + \hmax)}{2 \hmax + \mu_2 - \mu_1} - \hmax \\
		&= \frac{\hmax (3 \mu_1 + \mu_2)}{2 \hmax + \mu_2 - \mu_1}
	\end{align*}
	First assume $\hmax \geq 2 \mu_1$, then
	\begin{align*}
		2 \hmax^2 &\geq 4 \hmax \mu_1 \\
		2 \hmax^2 + \hmax \mu_2 - \hmax \mu_1 &\geq 3 \hmax \mu_1 + \hmax \mu_2 \\
		 \hmax &\geq \frac{\hmax (3 \mu_1 + \mu_2)}{2 \hmax + \mu_2 - \mu_1}
	\end{align*}
	Now assume $2 \mu_1 \geq \hmax$, then
	\begin{align*}
		2 \mu_1 (\mu_2 - \mu_1) &\geq \hmax \mu_2 - \hmax \mu_1 \\
		2 \mu_1 (2 \hmax + \mu_2 - \mu_1) &\geq \hmax \mu_2 + 3 \hmax \mu_1 \\
		2 \mu_1 &\geq \frac{\hmax (3 \mu_1 + \mu_2)}{2 \hmax + \mu_2 - \mu_1} \qedhere
	\end{align*}
\end{proof}



%% file: example.tex
\section{Numerical Examples}
Given $\hmax$, $\lambda_2$ and $\lambda_N$, let $a,b,c,d$ represents the limit values of inequalities \eqref{ineq:main1} - \eqref{ineq:main3} as defined in the proof of Corollary \ref{cor:ineq_consistent}. The following algorithm is used to calculate $K$ that solves the consensus problem for the double integrator agent dynamics:
\begin{algorithm}
	\caption{Stabilizing controller design algorithm}
	\begin{algorithmic}[0]
		\State $\mu_1 \gets \hmax / 2$
		\State $\mu_2 \gets -\mu_1 + 2 \hmax \lambda_N / \lambda_2 + 1$ \Comment{See \eqref{ineq:necessary}}
		\State $\Delta k \gets 0.9 d$
		\State $k_1 \gets \left( \min\{a, b - \Delta k\} + \max\{ 0, c - \Delta k \} \right) / 2$
		\State $k_2 \gets k_1 + \Delta k$
		\State $K \gets \mat{k_1 & k_2} T^{-1}$
	\end{algorithmic}
\end{algorithm}

\begin{eexample}
	We consider the following topologies. The sampling intervals are generated randomly from $(0,3)$ and at every 50th step one of the following topologies are selected randomly. $\lambda_2 = 0.3$ and $\lambda_N = 6$ selected. $\max_{ij} \abs{x_i(t_k) - x_j(t_k)}$,  $i,j = 1,\dots,N$ is plotted against the time step $k$ for 100 simulations. The results are shown in Figure \ref{fig:ex1}.
	\begin{figure}[!ht]
		\centering
		\includegraphics[width=\columnwidth]{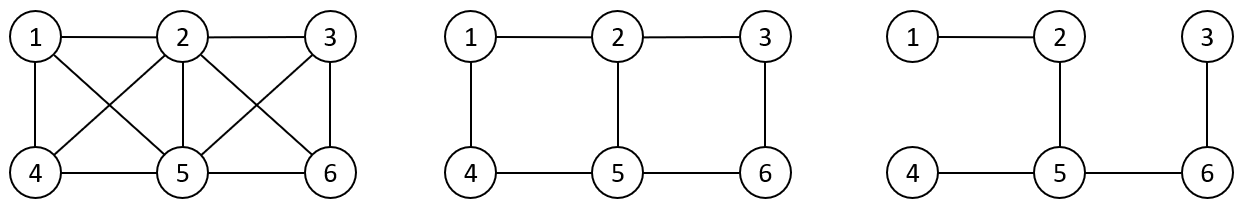}
	\end{figure}
	\begin{figure}[!ht]
		\centering
		\includegraphics[width=.9\columnwidth]{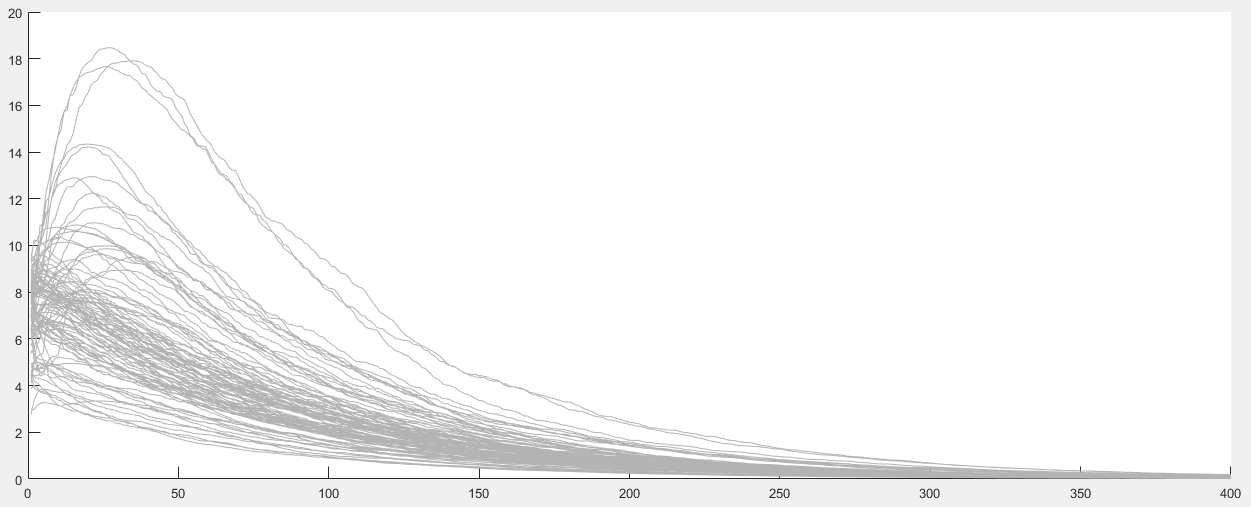}
		\caption{$\max_{ij} \abs{x_i(t_k) - x_j(t_k)}$ for the controller $K = \mat{0.0009 & 0.1093}$}
		\label{fig:ex1}
	\end{figure}
\end{eexample}

\begin{eexample}
	We consider 100 agents with random topology changing at every 50th step where $\lambda_2 = 5, \lambda_N = 60$ and $\hmax = 1$. $\max_{ij} \abs{x_i(t_k) - x_j(t_k)}, i,j = 1,\dots,N$ is plotted against the time step $k$ for 100 simulations. The results are shown in Figure \ref{fig:ex2}.
	\begin{figure}[!ht]
		\centering
		\includegraphics[width=.9\columnwidth]{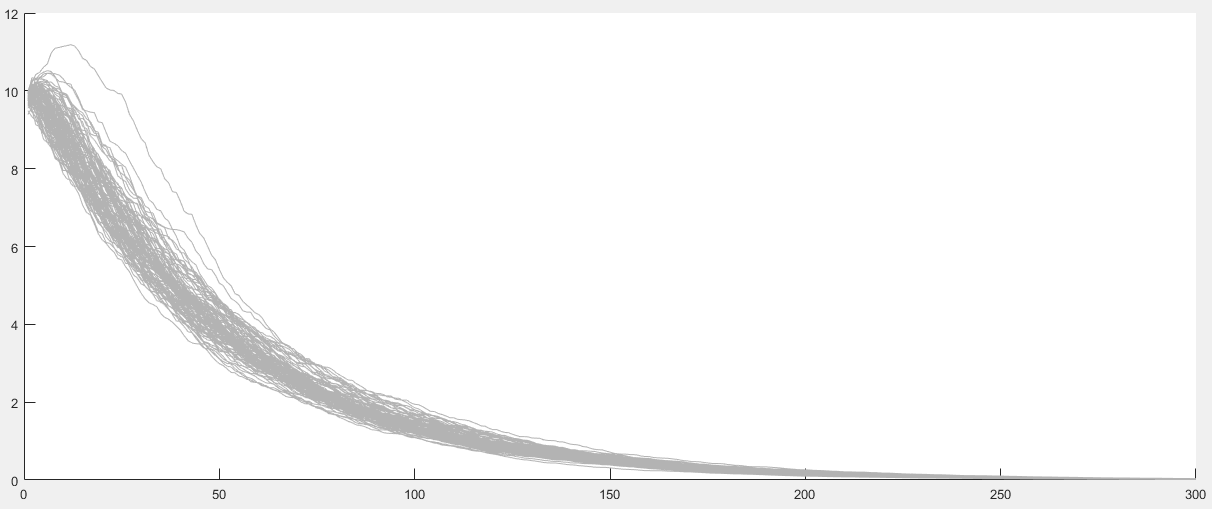}
		\caption{$\max_{ij} \abs{x_i(t_k) - x_j(t_k)}$ for the controller $K = \mat{0.0013 & 0.032}$}
		\label{fig:ex2}
	\end{figure}
\end{eexample}



%% file: conclusion.tex
\section{Conclusion}
A sufficient condition is given for checking whether a distributed state feedback control law solves the asymptotic consensus problem for the general linear dynamics in the case of nonuniform sampling. The condition is based on the stabilization problem of a specific subsystem, which is equivalent to the consensus problem. It depends on transforming the state coordinates, in which the amplitude of the state vector gets smaller at every step. It is shown that given condition immediately generalizes to the changing topology in the case of balanced networks.

Simple and explicit inequalities are given to design a controller in the case of double integrator dynamics and changing topology with balanced network graphs. These inequalities depend on the maximum sampling time and the interval that contains all eigenvalues of the changing topology graph Laplacian eigenvalues. It is also shown that such a controller always exist as long as the topology graphs have a spanning tree. Numerical examples are given to demonstrate the accuracy of the theoretical results.

It would be interesting to study the general graph case for double integrator systems using this method in the future. Also, it can possibly be extended when the measurements are transmitted asynchronously to the neighbor agents.